\documentclass[12pt,fleqn]{article}

\usepackage{amsmath}
\usepackage{amsfonts}
\usepackage{amsthm}

\newcommand{\Rset}{\mathbb{R}}
\newcommand{\F}{\mathcal{F}}
\newcommand{\Pra}{\mathcal{P}}

\begin{document}

\title{SELFDECOMPOSABILITY PERPETUITY LAWS AND STOPPING TIMES\footnote{Research supported by
the grant No. P03A 029 14 from KBN, Warsaw, Poland}}

\author{Zbigniew J. Jurek (Wroclaw, Poland)}

\date{see: \textbf{Probab. Math. Stat.} vol. 19, 1999, pp. 413-419}

\maketitle

\newtheorem{thm}{THEOREM}
\newtheorem{lem}{LEMMA}
\newtheorem{prop}{PROPOSITION}
\newtheorem{cor}{COROLLARY}

\theoremstyle{remark}
\newtheorem{rem}{REMARK}

In the probability theory limit distributions (or probability
measures) are often characterized by some convolution equations
(factorization properties) rather than by Fourier transforms (the
characteristic functionals). In fact, usually the later follows
the first one. Equations, in question, involve the
multiplication by the positive scalars $c$ or an action of the
corresponding dilation $T_{c}$ on measures. In such a setting, it
seems that there is no way for stopping times (or in general, for
the stochastic analysis) to come into the ``picture''. However, if
one accepts the view that the primary objective, in the classical
limit distributions theory, is to describe the limiting distributions
(or random variables) by the tools of random integrals/functionals
then one can use the stopping times. In this paper we illustrate such
a possibility in the case of selfdecomposability random variables
(i.e. L\'evy class L) with values in a real separable Banach space.
Also some applications of our approach to perpetuity laws are
presented; cf. \cite{Du1}, \cite{ED1}, \cite{GG1}. In fact, we show
that all selfdecomposable distributions are perpetuity laws.
\vspace{1cm}

\textbf{1.}
Let E be a real separable Banach space. An E-valued random variable
(rv) $X$, defined on a complete probability space $(\Omega,\F,\Pra)$, is said to
be \emph{selfdecomposable} (or \emph{a L\'evy class} L) if for
each $t>0$ there exists a rv $X_{t}$ independent of $X$ such that
\begin{equation}
X\stackrel{d}{=}X_{t}+e^{-t}X.\label{jr1}
\end{equation}
where `` $\stackrel{d}{=}$ '' means equality in distribution.
\medskip

\noindent
\textbf{Remark 1.} (i) The class of selfdecomposabile rv's (distributions)
coincides with the class of limiting rv's of the following
infinitesimal triangular arrays:
\[
a_{n}(Z_{1}+Z_{2}+\ldots +Z_{n})+b_{n},
\]
where $a_{n}>0$, $b_{n}\in \mathrm{E}$ and $Z_{1},Z_{2},\ldots$ are
independent $\mathrm{E}$-valued rv's; cf. for instance \cite{JM1},
Chapter 3.\\
(ii) In case of i.i.d. $Z_{n}$'s one gets in the above scheme the
class of all stable distributions.\\
(iii) In terms of probability distributions equation (\ref{jr1})
reads that for each $0<c<1$ there exists a probability measure
$\mu_{c}$ such that
\begin{equation}
\mu=\mu_{c}*T_{c}\mu.
\end{equation}
where `` $*$ '' denotes the convolution of measures and
$(T_{c}\mu)(\ \cdot\ )=\mu(c^{-1}\ \cdot\ )$. In other words, $T_{c}\mu$
is the image of a measure $\mu$ under the linear mapping $T_{c}\colon
\mathrm{E}\to \mathrm{E}$ given by $T_{c}x=cx$, $x\in \mathrm{E}$.
\medskip

Let us also recall that a stochastic base is an increasing and right
continuous family of $\sigma$-fields $\F_{t}\subset \F$; (i.e.
$\F_{s}\subset \F_{t}$ for $s<t$ and $\F_{t}=\bigcap_{s>t}\F_{s}$).

Furthermore, any mapping $\tau\colon \Omega\to [0,\infty)$ such that
\begin{equation}
\{\omega:\tau(\omega)\le t\}\in \F_{t},\ t\ge 0.
\end{equation}
is called \emph{a stopping time}.

A family $Y(t)$, $t\ge 0$, of E-valued random variables is called
\emph{a L\'evy process} provided
\begin{itemize}
\item
$Y(0)=0$ $\Pra$-a.s., $Y(t+s)-Y(s)\stackrel{d}{=}Y(t)$
for all $s,t\ge 0$,
\item
$Y(t_{k})-Y(t_{k-1})$, $k=1,2,\ldots,n$,
\emph{are independent} for all $0\le t_{0}<\ldots<t_{n}$, $n\ge 1$,
\item
$t\mapsto Y(t,\omega)$ \emph{are cadlag functions, for}
$\Pra$-\emph{a.a.} $\omega\in \Omega$.
\end{itemize}

Of course, for any $c>0$ and a L\'evy process $Y$, one has that
$Y_{c}(t):=Y(t+c)-Y(t)$, $t\ge 0$, is a new L\'evy process with
$Y_{c}\stackrel{d}{=}Y$ in the Skorohod space
$D_{E}[0,\infty)$ of all cadlag functions. Moreover, $Y_{c}$ is
independent of the $\sigma$-field $\sigma\{Y(t):0\le t\le
c\}=\F_{c}^{Y}$ In fact, for any stopping $\tau$ with respect to
$\F_{t}^{Y}$ and has that
\begin{equation}
Y_{\tau}(t):=Y(t+\tau)-Y(\tau),\ t\ge 0,\label{jr2}
\end{equation}
is a L\'evy process such that $Y_{\tau}\stackrel{d}{=}Y$ and $Y_{\tau}$
is independent of the $\sigma$-field $\F_{\tau}$ defined as follows
\begin{equation}
\F_{\tau}=\{A\in \F:A\cap[\tau\le t]\in \F_{t}\ \mbox{for each}\ t\ge 0\};
\end{equation}
cf. for instance \cite{Bi1}, Theorem 32.5 to derive the above
statements for any L\'evy process.
\vspace{0.5cm}

Here is the main result which extends (\ref{jr1}) for some stopping
times.

\begin{thm}
Suppose $X$ is a selfdecomposable $\mathrm{E}$-valued rv. Then there
exists a stochastic base $(\F_{t})_{t\ge 0}$ such that for each
$\F_{t}$-stopping time $\tau$ there are independent\label{jth1}
$\mathrm{E}$-valued rv's $X_{\tau}$ and $X'$ satisfying
\begin{description}
\item[(i)] $X\stackrel{d}{=}X_{\tau}+e^{-\tau}X'$;
\item[(ii)] $X'$ is independent of $X$ and $X'\stackrel{d}{=}X$;
\item[(iii)] vector $(X_{\tau},e^{-\tau})$ is independent of $X'$.
\end{description}
\end{thm}

\begin{proof}
From \cite{Ju1} or \cite{JM1} p.124 we conclude that rv $X$ is
selfdecomposable (i.e. (\ref{jr1}) holds) if and only if there exists
a unique, in distribution, L\'evy process $Y$ such that
$\mathbb{E}[\log(1+\|Y(1)\|)]<\infty$ and
\begin{equation}
X\stackrel{d}{=}\int_{(0,\infty)}e^{-s}dY(s).\label{jr4}
\end{equation}
(We refer to $Y$ as the \textbf{b}ackground \textbf{d}riving
\textbf{L}\'evy \textbf{p}rocess of $X$; in short: Y is BDLP for
$X$, cf. \cite{Ju2}). Taking $\F_{t}=\sigma(Y(s):s\le t)$ and defining
\begin{equation}
Z(t):=\int_{(0,t]}e^{-s}dY(s)\equiv e^{-t}Y(t)+\int_{(0,t]}Y(s-)e^{-s}ds,
\label{jr3}
\end{equation}
we have that $Z(\tau(\omega))$ is $\F_{\tau}$-measurable for a stopping
time $\tau$; cf. \cite{Sh1}, p. 18--20. Finally using (\ref{jr2}) and
(\ref{jr3}) one gets
\begin{multline}
X \stackrel{d}{=} \int_{(0,\tau]}e^{-s}dY(s)+
\int_{(\tau,\infty)}e^{-s}dY(s)={}\\
= Z(\tau)+e^{-\tau}\int_{(0,\infty)}e^{-s}dY_{\tau}(s)=
X_{\tau}+e^{-\tau}X',\label{jr6}
\end{multline}
with $X_{\tau}=Z(\tau)$ independent of
$X'=\int_{(0,\infty)}e^{-s}dY(s)\stackrel{d}{=}X$ because $Y_{\tau}$
is independent of $\F_{\tau}$. This completes the proof of Theorem 1.
\end{proof}

\noindent
\textbf{Remark 2.} The integral in (\ref{jr4}) is defined as a limit
of $Z(t)$, given by (\ref{jr3}), as $t\to\infty$. Existence of the
limit (in probability, a.s., or in distribution) is equivalent to the
condition $\mathbb{E}[\log(1+\|Y(1)\|)]<\infty$; cf. \cite{Ju1} or
\cite{JM1} p.122.

\begin{cor}
Let $Y$ be a BDLP of a selfdecomposable rv $X$ and let
$\F_{t}=\sigma(Y(s):s\le t)$, $t\ge 0$ be the stochastic base given
by $Y$. Then for each $\F_{t}$-stopping time $\tau$ there exists a rv
$X_{\tau}$ independent of $X$ such that
\begin{equation}
X\stackrel{d}{=}X_{\tau}+e^{-\tau}X.\label{jr5}
\end{equation}
\end{cor}
\vspace{0.5cm}

The equality in distribution in (\ref{jr5}) can be strengthen as
follows.

\begin{cor}
Let $Y$ be $\mathrm{E}$-valued L\'evy process such that\\
$\mathbb{E}[\log(1+\|Y(1)\|)]<\infty$ and $(\F_{t})_{t\ge 0}$ be the
natural filtration given by $Y$. Then for any $\F_{t}$-stopping time $\tau$
one has
\begin{multline}\label{jr7}
\int_{(0,\infty)}e^{-s}dY(s,\omega)=
\int_{(0,\tau(\omega)]}e^{-s}dY(s,\omega)+{}\\
+e^{-\tau(\omega)}\int_{(0,\infty)}e^{-s}dY(s+\tau(\omega),\omega)
\end{multline}
for $\Pra$-a.a. $\omega\in \Omega$.
\end{cor}

\begin{proof}
This is a consequence of Remark 2 and the equality in the formula
(\ref{jr6}).
\end{proof}

It is well-know that any L\'evy process $Y$ can be written as a sum of
two independent L\'evy processes $Y^{c}$ and $Y^{d}$, i.e.,
$Y=Y^{c}+Y^{d}$, where $Y^{c}$ is purely continuous (Gaussian) while
$Y^{d}$ is purely discontinuous (cadlag) process. Furthermore, for a
Borel subset $A$ separated from zero
(i.e., $A\subset \{x\in E:\|x\|\ge \epsilon\}$ for some $\epsilon>0$) we
define
\[
Y^{d}(t;A):=\sum_{0<s\le t}\Delta Y^{d}(s)1_{A}(\Delta Y^{d}(s)),
\]
where the jumps $\Delta Y^{d}(s):=Y^{d}(s)-Y^{d}(s-)$ are in the set
$A$, which a L\'evy process independent of the process
$Y^{d}(t)-Y^{d}(t;A)$. All the above allows us to have the following:

\begin{cor}
\begin{description}
\item[(i)]
Let $Y^{d}$ be a purely discontinuous L\'evy process with finite
logarithmic moment and \label{jcor1}
\[
\tau_{0}=\inf\{t>0: Y^{d}(t)\ne 0\}
\]
be the stopping time of the first non-zero value. Then
\begin{multline}\label{jr9}
\int_{(0,\infty)}e^{-s}dY^{d}(s)=e^{-\tau_{0}}Y^{d}(\tau_{0})+\\
e^{-\tau_{0}}\int_{(0,\infty)}e^{-s}dY^{d}(s+\tau_{0}),\ \Pra\mbox{-a.s.}
\end{multline}
\item[(ii)]
Let $\tau_{A}=\inf\{t>0: Y^{d}(t;A)\ne 0\}$ be the stopping time of
the first jump whose values is in $A$. Then
\begin{multline}
\int_{(0,\infty)}e^{-s}dY^{d}(s;A)=e^{-\tau_{A}}Y^{d}(\tau_{A};A)+{}\\
+e^{-\tau_{A}}\int_{(0,\infty)}e^{-s}dY^{d}(s+\tau_{A};A),\ \Pra\mbox{-a.s.}
\label{jr8}
\end{multline}
\end{description}
\end{cor}

\begin{proof}
Apply the above stopping times in the equation (\ref{jr7}).
\end{proof}

\noindent
\textbf{Remark 3.} (a) Random integrals appearing in (\ref{jr8}) are
independent, identically distributed and selfdecomposable. Similarly
holds for integrals in (\ref{jr9}) and the outmost integrals in
(\ref{jr7}).\\
(b) If $\tau_{1}=\tau_{A}$ and $\tau_{k}$, $k\ge 1$, are the
consecutive random times of the jumps of the process $Y^{d}(t;A)$
with $\tau_{k}\uparrow +\infty$, a.e., then one gets factorization
\begin{equation}
\int_{(0,\infty)}e^{-s}dY^{d}(t;A)=\sum_{k=1}^{\infty}e^{-\tau_{k}}
\Delta Y(\tau_{k};A),\ a.e.,
\end{equation}
where $\Delta Y(\tau_{k};A)$  are independent of $\tau_{k}-\tau_{k-1}$
for $k\ge 1$.
\vspace{1cm}

\textbf{2.} In this section we consider only \emph{real} valued
random variables. Let $(A,B)$, $(A_{1},B_{1})$, $(A_{2},B_{2}),\ldots$ be
a sequence of i.i.d. random vectors in $\Rset^{2}$ which define the
stochastic difference equation
\begin{equation}
Z_{n+1}=A_{n}Z_{n}+B_{n},\ n\ge 1.\label{jr10}
\end{equation}
Equation (\ref{jr10}) appear in modelling many real situations
including economics, finanse or insurance; cf. for instance \cite{ED1},
\cite{GG1} and the reference there. One many look at (\ref{jr10}) as
an iteration of the affine random mapping $x\mapsto Ax+B$, So,
starting with $Z_{0}$ and $(A_{0},B_{0})=(A,B)$ we get
\[
Z_{n+1}=A_{n}A_{n-1}\ldots A_{0}Z_{0}+
\sum_{k=0}^{n}B_{k}A_{k+1}A_{k+2}\ldots A_{n}.
\]
Putting $Z_{0}=0$ and assuming ($Z_{n}$) converges to $Z$ we get
\[
Z\stackrel{d}{=}\sum_{k=1}^{\infty}B_{k}\prod_{l=1}^{k-1}A_{l}.
\]
In insurance mathematics distributions of $Z$ are called \emph{perpetuities}.
Note that by (\ref{jr10}) perpetuities are the solution to
\begin{equation}
Z\stackrel{d}{=}AZ+B,\label{jr11}
\end{equation}
i.e., $Z$ is \emph{a distributional fixed-point} of the random affine mapping
$x\mapsto Ax+B$, $x\in \Rset$.

What triplets $A,B,X$ satisfy
(\ref{jr11}) with $(A,B)$ independent of $X$? Or are there
independent rv's $A,C,Z$ such that
\begin{equation}
Z\stackrel{d}{=}A(Z+C)\label{jr12}
\end{equation}
It seems that there are not to many explicite examples of (\ref{jr11})
or (\ref{jr12}); cf. \cite{Du1}, p.288. Results from previous section
can now be phrased as follows:

\begin{cor}
\begin{description}
\item[(i)]
All selfdecomposable distributions are perpetuities, i.e, satisfy
(\ref{jr11}) with non-trivial $0\le A\le 1$ a.s.
\item[(ii)]
All sefldecomposable distributions whose BDLP $Y$ have non-zero
purely discontinuous part, have convolution factors that satisfy the
equation (\ref{jr12}).
\end{description}
\end{cor}

Let $\gamma_{\alpha,\lambda}$ denotes a gamma rv with parameters
$\alpha>0$, $\lambda>0$, i.e., it has the probability density
\[
f_{\alpha,\lambda}=\frac{\lambda^{\alpha}}{\Gamma(\alpha)}x^{\alpha-1}e^{-\lambda x}
1_{(0,\infty)}(x).
\]
It is known, cf. \cite{Ju2}, \cite{Ju3}, that $\gamma_{\alpha,\lambda}$ is
selfdecomposable and its BDLP is given by $\lambda^{-1}Y_{0}(\alpha t)$,
where
\begin{equation}
Y_{0}(t)=\sum_{j=1}^{N(t)}\gamma_{1,1}^{(j)},
\end{equation}
$\gamma_{1,1}^{(1)}$, $\gamma_{1,1}^{(2)}\ldots$ are i.i.d. copies of
$\gamma_{1,1}$ and $N(t)$ is a standard Poisson process, i.e., it has
stationary, independent increments, $N(0)=0$ a.e. and for $t>s>0$
\[
\Pra[N(t)-N(s)=k]=e^{-(t-s)}\frac{(t-s)^{k}}{k!},\ k=0,1,2,\ldots\ .
\]
If $0<\tau_{1}<\tau_{2}<\ldots<\tau_{n}<\ldots$ are the consecutive
random times (arrival times) of the jumps of $N$ then
$\tau_{n}-\tau_{n-1}\stackrel{d}{=}\gamma_{1,1}$ for $n\ge 1$ are independent
and $\tau_{n}\stackrel{d}{=}\gamma_{n,1}$ for $n\ge 1$.

\begin{prop}
For gamma rv \ $\gamma_{\alpha,\lambda}$ \ one has
\begin{description}
\item[(i)]
\hfill$\gamma_{\alpha,\lambda}\stackrel{d}{=}
e^{-\gamma_{\alpha,1}}(\gamma_{1,\lambda}+\gamma_{\alpha,\lambda})
\stackrel{d}{=}U^{1/\alpha}\gamma_{\alpha+1,\lambda}$\hfill
\vspace{1ex}

where $U$ is uniformly distributed on $[0,1]$ independent of
$\gamma_{\alpha+1,\lambda}$ and the three middle rv's are independent too.
\item[(ii)]
\hfill$\gamma_{\alpha,\lambda}\stackrel{d}{=}
\displaystyle{\sum_{n=1}^{\infty}U_{1}^{1/\alpha}U_{2}^{1/\alpha}\ldots U_{n}^{1/\alpha}
\gamma_{1,\lambda}^{(n)}}$\hfill
\vspace{1ex}

where $\gamma_{1,\lambda}^{(1)}$, $\gamma_{1,\lambda}^{(2)}\ldots$ are i.i.d. copies of
$\gamma_{1,\lambda}$, $U_{1}$, $U_{2}\ldots$ are i.i.d. copies of
$U$ and both sequences are independent too.
\end{description}
\end{prop}

\begin{proof}
(i) Since $\gamma_{\alpha,\lambda}$ has BDLP
$Y(t)=\lambda^{-1}Y_{0}(\alpha t)$, therefore by Corollary \ref{jcor1}(i)
 and Remark 3(a) we have
\begin{eqnarray*}
\gamma_{\alpha,\lambda} &\stackrel{d}{=}& e^{-\gamma_{\alpha,\lambda}}
\gamma_{1,\lambda}^{(1)}+
e^{-\gamma_{\alpha,1}}\int_{(0,\infty)}e^{-s}dY(s+\tau_{1})=\\
&=& U^{1/\alpha}(\gamma_{1,\lambda}^{(1)}+\tilde{\gamma}_{\alpha,\lambda}
\stackrel{d}{=} U^{1/\alpha}\gamma_{1+\alpha,\lambda}.
\end{eqnarray*}
(Equality of the two outmost terms in (i) can be also easily checked by
comparing the corresponding characteristic functions.)

(ii) Repeating the middle equality in (i) and using facts that
$\tau_{n}\uparrow +\infty$ a.s., and $Y(\cdot)$ is independent of
$Y(\cdot + \tau_{1}+\tau_{2}+\ldots+\tau_{k})-Y(\tau_{1}+\tau_{2}+\ldots+\tau_{k})$
one arrives at (ii).
\end{proof}
\bigskip

\textbf{3.} The metod of random integral representation is also applicable to
operator-selfdecomposable distributions; cf. \cite{JM1} Chapter 3, or
\cite{Ju1}. Recall that a Banach space $\mathrm{E}$-valued rv $X$ is
$Q$-selfdecomposable if for each $t>0$ there exists a rv $X_{t}$ independent
of $X$ such that
\begin{equation}
X\stackrel{d}{=} X_{t}+e^{-tQ}X;
\end{equation}
$Q$ is a bounded linear operator on $\mathrm{E}$ and $e^{-t Q}$ is the operator
given by a power series.

\begin{thm}
Suppose that $X$ is $Q$-decomposable $\mathrm{E}$-valued rv and $e^{-tQ}\to 0$,
as $t\to \infty$, in the norm topology. Then there is a filtration
$(\F_{t})_{t\ge 0}$ such that for each stopping times $\tau$ there exist
independent $\mathrm{E}$-valued rv's $X_{\tau}$ and $X'$ satisfying
\begin{description}
\item[(i)] $X\stackrel{d}{=} X_{\tau}+e^{-\tau Q}X'$;
\item[(ii)] $X'$ is independent copy of $X$;
\item[(iii)] the random vector $(X_{\tau},e^{-\tau Q})$ is independent of $X'$.
\end{description}
\end{thm}

\begin{proof}
From \cite{Ju1} or \cite{JM1} Chapter III we have that $X$ is $Q$-selfdecomposable
if and only if
\begin{equation}
X\stackrel{d}{=}\int_{(0,\infty)} e^{-tQ}dY(t)\label{jr13}
\end{equation}
for a uniquely defined L\'evy process $Y$ such that
$\mathbb{E}[\log(1+\|Y(t)\|]<\infty$. So, (\ref{jr13}) allows us to proceed as
in the proof of Theorem \ref{jth1}.
\end{proof}

\noindent
\textbf{Remark 4.} (a) Corollaries from Section 1 have their ``operator''
counterparts.\\
(b) For a given $\mathrm{E}$-valued rv $B$ and a random bounded linear operator
$A$ on a Banach space $\mathrm{E}$, consider the affine random mapping
$x\mapsto Ax+B$. The question of finding all distributial fix-points, i.e.,
all $\mathrm{E}$-valued rv's $X$ such that
\begin{equation}
X\stackrel{d}{=} AX+B,\label{jr14}
\end{equation}
seems, to be more difficult as the composition of operators is not
commutative. However, random integrals of the form
\begin{equation}
\int_{(a,b]}f(t)dY(r(t)),
\end{equation}
where $Y$ is a L\'evy process, $r(t)$ is a change of time, $f$ is a process
or deterministic funtions, \emph{might provide a tool} of constructing $X$
satisfying the equation (\ref{jr14}) or its variants
(like (\ref{jr12})). The present paper illustrates this approach in a case
of the selfdecomposable distributions and their random integral
representations.

\medskip
\textbf{Added in proof.} Corollary 1 is also true when the
stopping time $\tau$ is replaced by a non-negative random varaible
$T$ independent of the BDLP $Y$.

\noindent
Address:\\
Institute of Mathematics\\
The University of Wroc\l aw\\
pl. Grunwaldzki 2/4\\
50-384 Wroc\l aw, Poland\\
(zjjurek@math.uni.wroc.pl)

\end{document}